\documentclass[reqno, 12pt]{amsart}
\usepackage[utf8]{inputenc}
\usepackage[T1]{fontenc}
\usepackage[english]{babel}
\usepackage[babel]{microtype}
\usepackage[a4paper, margin=1.2in]{geometry}
\usepackage{amsmath,amssymb,amsthm,amsrefs}
\usepackage{mathrsfs,mathtools}
\usepackage{commath}
\usepackage{bbm}
\usepackage{dsfont}
\usepackage{lmodern}
\usepackage{cases}
\usepackage{enumitem}
\usepackage{centernot}
\usepackage{todonotes}
\usepackage{graphicx}
\usepackage{url}

\numberwithin{equation}{section}
\newtheorem{thm}    {Theorem}

\setlist[enumerate,1]{label=(\roman*)}

\DeclareMathOperator{\C}{\mathbb{C}}

\newcommand{\ds}{\displaystyle}

\begin{document}

\title{On Graph Continued Fractions and the Heilmann-Lieb Theorem}

\author[T. J. ~Spier]{Thomás Jung Spier \\ \today}
\address{IMPA, Rio de Janeiro, RJ, Brasil}
\email{thomasjs@impa.br}


\begin{abstract}
Inspired by Viennot's observation that matching polynomials are numerators of branched continued fractions we present a proof of the Heilmann-Lieb Theorem.  
\end{abstract}


\clearpage\maketitle
\thispagestyle{empty}


\section{Introduction}

Let $G$ be the complete graph with vertex set $[n]$. Define variable weights $x_i$ and non-negative weights $\lambda_{jk}$ for each of the vertices and edges, respectively. Considering edges with weight set to zero as non-existent this definition captures all graphs.

A {\it matching} in $G$ is a set of edges, no two of which have a vertex in common, together with their respective endpoints. Denote by $\mathcal{M}_G$ the set of all matchings of $G$. Then the {\it matching polynomial} of $G$ is

\[\mu(G):=\ds\sum_{M\in\mathcal{M}_G}\ds\prod_{i\not\in M}x_i\ds\prod_{jk\in M}\lambda_{jk}.\]  

This is a real multivariate polynomial in the $n$ vertex variables $x_i$. It is also convenient to define $\mu(\emptyset)=1$.

The matching polynomial was first considered in statistical physics by Heilmann and Lieb \cite{heilmann-lieb}. In their article they proved:

\begin{thm}(Heilmann-Lieb \cite{heilmann-lieb})\label{Heilmann-Lieb} The matching polynomial of $G$ is different from zero if one of the following conditions is satisfied: 
	\begin{itemize}
		\item $Re(x_i)>0$ for every $i$;
		\item $|x_i|>2\sqrt{B_G}$ for every $i$, where $B_G$ is equal to $\ds\max_j\,\max_{\mathclap{\substack{A\subseteq [n]\setminus j\\ |A|=n-2}}}\quad\sum_{k\in A}\lambda_{jk}$ if $n\geq 3$, and equal to $\lambda_{12}/4$ or $0$ if $n$ is two or one, respectively.
	\end{itemize}	
\end{thm}

A survey of the history of this polynomial is available in Gutman \cite{gutman2016survey}. 
	
In this short note our aim is to show how Theorem \ref{Heilmann-Lieb} is naturally connected to continued fractions. This follows from Viennot's \cite{viennot1985combinatorial} observation that matching polynomials are numerators and denominators of branched continued fractions. Even though we present a restatement of the original proof of Theorem \ref{Heilmann-Lieb} we believe that this different perspective can be useful. The connection between matching polynomials and continued fractions will be further explored in \cite{graphcf}.


\section{Proof of main result}

Notice that for every rooted tree one can associate a branched continued fraction in a natural way, as exemplified in Figure \ref{branchedcf}. We call this a {\it tree continued fraction}.

\begin{figure}[h]
	\includegraphics[width=\linewidth]{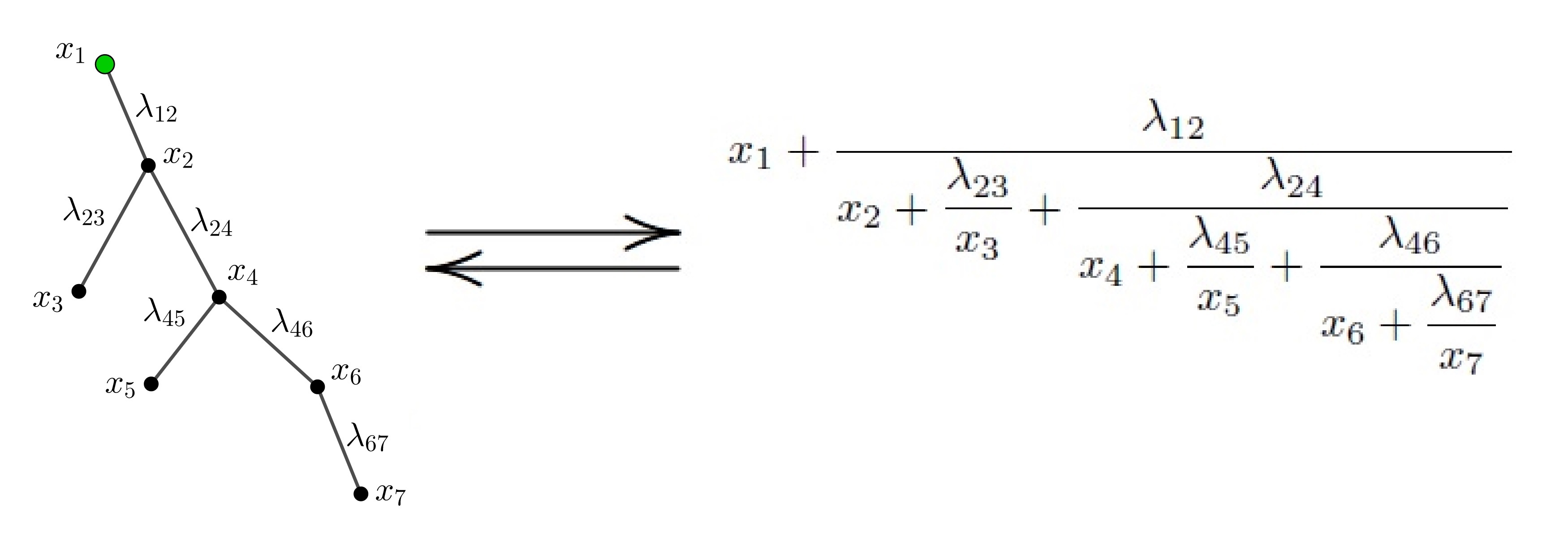}
	\caption{A rooted tree and its associated tree continued fraction.}
	\label{branchedcf}
\end{figure}

If we write $T$ for the tree and $i$ for its root, then the associated tree continued fraction is equal to $\dfrac{\mu(T)}{\mu(T\setminus i)}$. This fact can be proved using a recurrence for the matching polynomial. For every graph $G$ and vertex $i$, if we separate the matchings of $G$ into those that cover, or not, the vertex $i$ we obtain the recurrence,

\[\mu(G)=\ds\sum_{j\neq i}\lambda_{ij}\mu(G\setminus i,j)+x_i\mu(G\setminus i)\iff \dfrac{\mu(G)}{\mu(G\setminus i)}=x_i+\ds\sum_{j\neq i}\dfrac{\lambda_{ij}}{\dfrac{\mu(G\setminus i)}{\mu(G\setminus i,j)}}.\]

To finish the proof of the observed fact, we substitute the tree $T$ for the graph $G$ in this last equation and iterate the recurrence.

Looking at this procedure one can see that in principle it should work more generally for every rooted graph, the only missing ingredient being the analog of a tree continued fraction. Iterating the recurrence for a rooted graph, what one obtains at the end is a tree continued fraction for the {\it rooted tree of paths} of the rooted graph.

For a rooted graph $G$ with root $i$ its {\it rooted tree of paths} $T^i_G$ is the rooted tree with vertices labeled by paths in $G$ starting at $i$, where two vertices are connected if one path is a maximal sub-path of the other. The root of $T^i_G$ is the trivial path $i$, and the weights of $T^i_G$ are obtained from the weights of $G$, as exemplified in Figure \ref{godsil}.
 
This motivates the following definition. Given a rooted graph $G$ with root $i$ define its {\it graph continued fraction} as $\alpha_i(G):=\dfrac{\mu(G)}{\mu(G\setminus i)}$. Notice that this is consistent with the definition of tree continued fraction. The observation above leads to the equality $\alpha_i(G)=\alpha_i(T^i_G)$, originally due to Godsil \cite{godsil_matchings}. An illustration of this equality is presented in Figure \ref{godsil}, where, for simplicity, the rooted graphs represent their graph continued fractions. With the facts above, we are ready to prove the main theorem.

\begin{figure}[h]
	\includegraphics[width=\linewidth]{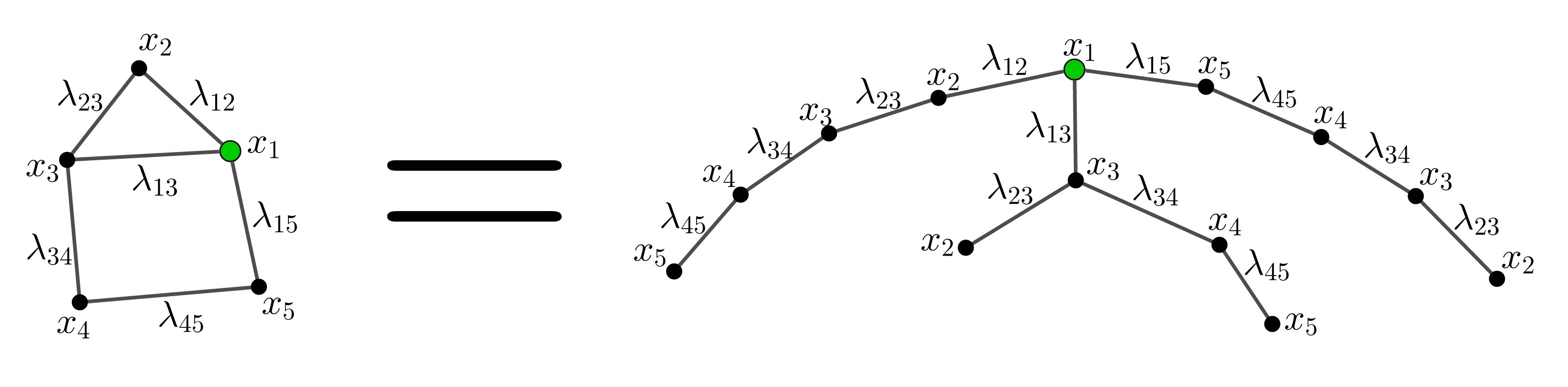}
	\caption{An illustration of the equality $\alpha_i(G)=\alpha_i(T^i_G)$.}
	\label{godsil}
\end{figure}

\begin{proof}[Proof of Theorem \ref{Heilmann-Lieb}.] The approach is the same as in \cite{heilmann-lieb}. Consider a graph $G$ and let $R$ be any of the two regions $[Re(x)>0]$ or $[|x|>2\sqrt{B_G}]$ in $\C$. Our aim is to prove that $\mu(G)$ is different from zero in $R^n$. Notice that for a graph with only one vertex this result is trivial. Assume, by induction hypothesis, that the statement is true for any graph with less vertices than $G$.
	
Choose as a root of $G$ any vertex $i$. By the induction hypothesis, and  $B_G\geq B_{G\setminus i}$, it is sufficient to prove that the graph continued fraction $\alpha_i(G)=\dfrac{\mu(G)}{\mu(G\setminus i)}$ is different from zero in $R^n$.

Recall that $\alpha_i(G)$ is equal to the tree continued fraction $\alpha_i(T^i_G)$. Following the structure of the rooted tree $T^i_G$, one can write $\alpha_i(G)=\alpha_i(T^i_G)$ as a composition of some functions

\[f_{j,A}(x_1,\dots,x_n):=x_j+\ds\sum_{k\in A}\frac{\lambda_{jk}}{x_k}, \]

with $j$ in $[n]$ and $A$ a subset of $[n]\setminus j$. Each function corresponding to a vertex in the rooted tree $T^i_G$. Notice that except for the last function in this composition, which corresponds to the root of $T^i_G$, all the other functions $f_{j,A}$ satisfy $|A|\leq n-2$. This can be seen by carefully examining the examples of Figures \ref{branchedcf} and \ref{godsil}.

Finally, observe that the image of $R^n$ by every function $f_{j,A}$ with $|A|\leq n-2$ is again contained in $R$, and that every function $f_{j,A}$ with $|A|=n-1$ is different from zero in $R^n$. Putting it all together it follows that $\alpha_i(G)=\alpha_i(T^i_G)$ is different from zero in $R^n$, which finishes the proof.

\end{proof}

\IfFileExists{references.bib}
  {\bibliography{references}}
  {\bibliography{../references}}


\end{document}